\documentclass[11pt]{article}

\usepackage[a4paper, total={6in, 8in}]{geometry}
\usepackage{url}
\usepackage[charter]{mathdesign}
\usepackage[T1]{fontenc}

\usepackage{amsmath,amsthm,amsfonts}
\usepackage{enumitem}
\usepackage{mathtools}
\usepackage{comment}
\usepackage{authblk}
\makeatletter
\title{The intersection of the subgroups of finite $p$-index in a multiple HNN-extension of an infinite cyclic group}

\author{V. Metaftsis} 
\author{D. Tsipa\thanks{%
		\protect\parbox[c]{0.18\linewidth}{%
			\protect\includegraphics[width=\linewidth]{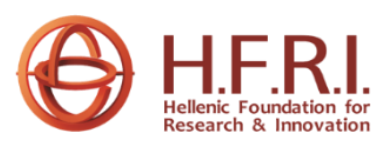}%
		}%
		\hfill%
		\protect\parbox[c]{0.74\linewidth}{%
				The research work was supported by the Hellenic Foundation for
			Research Innovation (HFRI) under the 3rd Call for HFRI PhD
			Fellowships. (Fellowship Number: 5161)%
		
		}%
}}
\affil[ ]{Department of Mathematics, University of the Aegean, Karlovassi, 832 00 Samos, Greece}
\affil[ ]{\textit {vmet@aegean.gr, dtsipa@aegean.gr}}
\date{}

\usepackage[overload]{empheq}
\theoremstyle{theorem} 
\newtheorem{theorem}{Theorem}[section]
\newtheorem{corollary}[theorem]{Corollary}
\newtheorem{lemma}[theorem]{Lemma}
\newtheorem*{lemma*}{Lemma}
\newtheorem*{theorem*}{Theorem}
\newtheorem*{corollary*}{Corollary}

\newtheorem*{proposition*}{Proposition}
\theoremstyle{definition}

\newtheorem*{example*}{Example}    

\newtheorem*{remark*}{Remark}

\newcommand{\longeq}{\scalebox{2}[1]{=}}

\newcommand{\N}{\mathbb{N}}
\newcommand{\Z}{\mathbb{Z}}

\newcommand{\om}{\omega}

\newcommand\myex{\mathrel{\stackrel{\makebox[0pt]{\mbox{\normalfont\scriptsize($3$)}}}{\longeq}}}
\def\g{\gamma}

\def\e{\epsilon}

\def\om{\omega}
\def\k{\kappa}
\def\l{\lambda}
\def\BS{\mathrm{BS}}
\def\w{\omega}
\def\<{\langle}
\def\>{\rangle}

\begin{document}
		\maketitle
		\vspace{-5mm}
		\begin{abstract}
		Let $G$ be a multiple HNN-extension of an infinite cyclic group. We will calculate the intersection  $(N_{p}){_\om}(G)$ of the normal subgroups of finite $p$-index in $G$  thus generalizing the result of Moldavanskii in \cite{Mold-ru} for Baumslag-Solitar groups. As a corollary we give necessary and sufficient conditions for $G$ to be residually finite $p$-group.
	\end{abstract}

\section{Introduction}

Let $G$ be a group and $\mathcal{P}$ be a  property. We say that $G$ is residually $\mathcal{P}$ if  for every non-trivial element $g\in G$ there is a homomorphism $f$ from $G$ to a group with property $\mathcal{P}$ such that the image of $g$ under $f$ is also non-trivial. In the present note we consider groups with $\mathcal{P}$ to be finite $p$-group. Then the above definition is equivalent to the following. A group $G$ is residually finite $p$- group if and only if $(N_{p})_\omega(G)$ is trivial, where  $(N_{p})_\omega(G)$ denotes the intersection of all finite index normal subgroups of $G$ with index some power of the prime number $p$.

The purpose of this note is to calculate the intersection $(N_{p})_\omega(G)$, when $G$ is a multiple HNN-extension of an infinite cyclic group. This type of group can also be considered as a Generalized Baumslag-Solitar group where the underlying graph is a rose graph (a multiple loop).
Generalized Baumslag-Solitar groups (GBS-groups) are fundamental groups of finite graphs of groups with infinite cyclic vertex and edge groups. These groups have interesting group-theoretic properties and have been the subject of recent research (see for example \cite{Delga, Fore,LevittA, Robi}).

In case the underlying graph is a simple loop we have Baumslag-Solitar groups that is,  $1$-relator groups of the form
$$BS(n,m)=\langle t, a \hspace{1mm}|\hspace{1mm}ta^{n}t^{-1}=a^{m}\rangle$$
where $m, n$ are non-zero integers. 

When $|n|=|m|=1$ , $BS(n,m)$ is called an  elementary Baumslag-Solitar group and is the fundamental group of the torus or the Klein bottle. These groups are well understood and can therefore be regarded as standing somewhat apart from the remaining Baumslag-Solitar groups. Moreover $\BS(n, m)$, $\BS(m, n)$ and $\BS(-n, -m)$ are pairwise isomorphic, we may assume, without loss of generality, that the integers $n$ and $m$ in the presentation of $\BS(n, m)$ satisfy the condition $0 < n \leq |m|$. For the rest of the paper we assume that each Baumslag-Solitar group $BS(n_{i},m_{i})$ is non-elementary and $0 < n_{i} \leq |m_{i}|$.

 Baumslag-Solitar groups seem to have first appeared in the literature in \cite{BS} as they were defined by Gilbert Baumslag and Donald Solitar in order to provide examples of non-Hopfian groups. These groups have played a really useful role in combinatorial and geometric group theory. In several situations they have provided examples which mark boundaries between different classes of groups and also they often provide a test for various theories.

 Moldavanskii in \cite{Mold-ru} (see also \cite{Mold}) proved the following.
\begin{theorem*}\label{Molda}[Moldavanskii]\ \\
	Let $G=BS(m,n)$, $p$ be a prime number and let $m=p^rm_1$ and $n=p^sn_1$ where $r,s\ge 0$ and $m_1,n_1$ are not divisible by $p$. Let also $d={\rm gcd}(m_1,n_1)$, $m_1=du$ and $n_1=dv$. Then
	\begin{enumerate}
		\item If $r\neq s$ or if $m_1\not\equiv n_1 \pmod p$, then $(N_{p})_{\om}(G)$ coincides with the normal closure of $a^{p^\xi}$ in $G$,  where $\xi=\min\{r,s\}$.
		\item If $r=s$ and $m_1\equiv n_1 \pmod p$, then $(N_{p})_{\om}(G)$ coincides with the normal closure  of the set $\{t^{-1}a^{p^ru}ta^{-p^{r}v}, [t^ka^{p^r}t^{-k},a]: k\in{\mathbb Z}\}$ in $G$.
	\end{enumerate}
\end{theorem*}
The following corollary characterizes exactly when a Baumslag-Solitar group $BS(m,n)$ is a residually finite $p$-group.

\begin{corollary}\label{BS}
	The Baumslag-Solitar group $BS(m,n)$ is residually finite $p$-group for some prime number $p$ if and only if $m=1$ and $n\equiv 1\pmod p$ or $n=m$ and $m=p^r$ or $p=2$ and $n=-m=-2^r$, $r\ge 1$.
\end{corollary}
\begin{remark*}
	This result also follows from the study of Kim and McCarron in certain one relator groups (see \cite[Main Theorem]{kimc}).  
\end{remark*}
Our aim is to generalize the work done in  \cite{Mold-ru} and calculate the intersection  $(N_{p})_{\om}(G)$ for the multiple HNN-extension
$$G=\langle a, t_{1}, t_{2},\dots,t_{r}\hspace{1mm}|\hspace{1mm}t_{i}a^{n_{i}}t_{i}^{-1}=a^{m_{i}}, i=1,\dots,r\rangle$$
where $n_{i}, m_{i},$ are non-zero integers.
We actually prove the following. 

If $p$ is a prime number then let $m_{i}=p^{\sigma_{i}}\hat{m}_{i},$ $n_{i}=p^{\tau_{i}}\hat{n}_{i}$, where $\hat{m}_{i},\hat{n}_{i}\in\Z$ with $p\nmid \hat{m}_{i},\hat{n}_{i}$ and $\sigma_{i},\tau_{i}\in \N$. 

Also, let $d_{i}=gcd(\hat{m}_{i},\hat{n}_{i})$. Then $\hat{m}_{i}=d_{i}u_{i}$ and  $\hat{n}_{i}=d_{i}v_{i}$ with $gcd(u_{i},v_{i})=1$.
We define $\theta_{i}$ to be \vspace{2mm}
\[ \theta_{i} = \left\{ \begin{array}{ll}
	min\{\sigma_{i},\tau_{i}\} & \mbox{if   $\sigma_{i}\neq \tau_{i}$ or ($\sigma_{i}=\tau_{i}$ and $p\nmid(\hat{m}_{i}-\hat{n}_{i})$})\\
	\infty &  \mbox{if } \sigma_{i}=\tau_{i}   \mbox{ and }  p|(\hat{m}_{i}-\hat{n}_{i}) \end{array} \right. \] \vspace{2mm}
for all $i\in\{1,\dots,r\}$.

\vspace{1mm}
\begin{theorem*}\label{Np}
	Let $G$ be the group $$G=\langle a, t_{1},\dots,t_{r}\hspace{1mm}|\hspace{1mm}t_{i}a^{n_{i}}t_{i}^{-1}=a^{m_{i}}, i\in\{1,\dots,r\}\rangle$$ with $n_{i}$, $m_{i}$ as described above. 
	\begin{enumerate}
	
		\item[(1)] If there exists $ i\in \{1,\dots,r\}$ such that $\theta_{i}\in\N$ then $(N_p)_{\omega}(G)=\langle a^{p^{\xi}} \rangle^G$, where $\xi=min\{\theta_{i}\}_{i=1}^{r}$.
		\item[(2)] If $\theta_{i}=\infty$\hspace{2mm}for all $ i\in \{1,\dots,r\}$ then $(N_p)_{\omega}(G)$ is the normal closure of the set
$$ \Big{\{}[\gamma_{2}(F_{r}),a^{p^{\Sigma}}], [t_{1}^{k_{1}}\cdots t_{r}^{k_{r}}at_{r}^{-k_{r}}\cdots t_{1}^{-k_{1}},a^{p^{\Sigma}}], t_{1}^{k_{1}}\cdots t_{r}^{k_{r}}a^{p^{\Sigma}\frac{y_{1}\cdots y_{r}}{\delta}} t_{1}^{-k_{1}}\cdots t_{r}^{-k_{r}}a^{-p^{\Sigma}\frac{\bar{y}_{1}\cdots \bar{y}_{r}}{\delta}}\Big{\}}$$
				 in $G$,  where $F_{r}$ is the free group generated by $t_{1},\dots,t_{r}$, $k_{i}\in \Z$, $\Sigma= \sum\limits_{i=1}^{r}\sigma_{i}$,  $\delta=\delta_{\{k_{i}\}_{i=1}^{r}}=gcd(y_{1}\cdots y_{r},\bar{y}_{1}\cdots \bar{y}_{r})$ and 
		
$$y_{i} = \left\{ \begin{array}{ll}
				u_{i}^{|k_{i}|} & \text{if}\hspace{3mm}   k_{i}>0 \\
				v_{i}^{|k_{i}|} &   \text{if}\hspace{3mm}   k_{i}<0\\
				1 &  \text{if}\hspace{3mm}  k_{i}=0 \end{array} \right.  \hspace{1cm} \& \hspace{1cm}  \bar{y}_{i} = \left\{ \begin{array}{ll}
				v_{i}^{|k_{i}|} & \text{if}\hspace{3mm}   k_{i}>0 \\
				u_{i}^{|k_{i}|} &   \text{if}\hspace{3mm}   k_{i}<0\\
				1 &  \text{if}\hspace{3mm}  k_{i}=0 \end{array} \right. $$

%
		for all $i\in\{1,\dots,r\}$.
	\end{enumerate}
\end{theorem*}
\vspace{3mm}
As a Corollary we provide necessary and sufficient conditions for $G$ to be residually finite $p$-group. Consider $r\geq2$.
\begin{corollary*}
	 The group $G$ is residually finite $p$-group for some prime number $p$ if and only if either  $n_{i}=m_{i}=p^{\sigma_{i}}$ for all  $i\in \{1,\dots,r\}$ or $p=2$ and $|n_{i}|=m_{i}=2^{\sigma_{i}}$  for all  $i\in \{1,\dots,r\}$.
\end{corollary*}
It is interesting to note here that the case where   $m=1$ and $n\equiv 1\pmod p$ is incompatible with the other two cases and so if it appears in a multiple HNN-extension it always gives elements in $(N_p)_{\omega}(G)$.
	\section{Auxiliary results}
		In this section we present some auxiliary results that we need in order to prove our main  Theorem.

	In  \cite{RV}  Raptis and Varsos proved the following.
	\begin{lemma}\label{RV}[Raptis, Varsos]\ \\
		The semi-direct product $G = K\rtimes H$ of a finite $p$-group $K$ by a residually finite $p$-group $H$ is residually finite $p$ if and only if the subgroup of the automorphisms of $K$ induced by the elements of $H$ forms a finite $p$-group of $AutK$. 
	\end{lemma}
	In the following Lemma we generalize the above result.
	\begin{lemma}\label{genRV}
		Let $K$ be a finite $p$-group and $\phi_{i}$ be isomorphisms between the subgroups $A_{i}$ and $B_{i}$ of $K$, $i\in \{1,\dots,n\}$. Let $G$ be the following multiple HNN-extension:
		$$G=\langle K, t_{i}\hspace{1mm}|\hspace{1mm}t_{i}A_{i}t_{i}^{-1}=\phi_{i}(A_{i})=B_{i}, \hspace{1mm}i\in \{1,\dots,n\} \rangle. $$
		If the isomorphisms $\phi_{i}$ can be extended to automorphisms $\theta_{i}$ of $K$ such that the subgroup $H$ of $AutK$ generated by $\theta_{i}$ is a finite $p$-group, then $G$ is residually finite $p$-group.
	\end{lemma}
	\begin{proof}
		Let $X=\langle r(K), H \rangle$ be the subgroup in the holomorph of $K$ generated by the right regular representation $r(K)$ of $K$ and the subgroup  $H=\<\theta_{1},\theta_{2},\dots,\theta_{n}\>$ of $AutK$. 
		
		Notice that the group $X$ is a finite $p$-group such that for each $i\in \{1,\dots,n\}$, $\theta_{i}r_{k}\theta_{i}^{-1}=r_{\theta_{i}(k)}$ for every $k\in K$.
		
		We define the map $\psi: G\to X$ that maps $k$ to $r_{k}$, for all $k\in K$ and $t_{i}$ to $\theta_{i}$, $i\in \{1,\dots,n\}$. One can easily see that $\psi$ is a homomorphism since it preserves the relations of $G$. 
		
		Since $Ker\psi\cap K=\{1\}$ and $G$ is a multiple HNN-extension we have that $Ker\psi$ is a free group. Thus $G$ is free-by-(finite $p$-group) and hence $G$ is a residually finite $p$-group. 
	\end{proof}
	\vspace{2mm}
	For the proof of the following Lemma we need Kummer's theorem (see \cite{Kum}) which  calculates the highest power of a prime number $p$ that divides a given binomial coefficient. This is called the $p$-adic valuation of the binomial coefficient and is denoted by $\nu_{p}{n \choose m} $.
	\begin{theorem*}\label{Kummer}[Kummer's Theorem]\ \\
		For given integers $n \geq m \geq  0 $ and a prime number $p$, the $p$-adic valuation $\nu_{p}{n \choose m} $ of the binomial coefficient ${n \choose m} $ is equal to the number of carries when $m$ is added to $n-m$ in base $p$.
	\end{theorem*}
	\begin{lemma}\label{aut}
		Let $G=\langle a,t \hspace{1mm}|\hspace{1mm}a^{p^{s}}=1 , ta^{n}t^{-1}=a^m\rangle$ , where $p$  is a prime number and $m,n$ are integers such that $p\nmid n,m$ and $p|(m-n)$. Then the automorphism  induced by $t$, has order a power of $p$. 
	\end{lemma}
	\begin{proof}
		Let $gcd(m-n,p^s)=p^r$. First we show that $[t,a^{p^{s-r}}]=1$.
		
		Since $gcd(m-n,p^s)=p^r$ we have  $m-n=cp^r$ and thus 
		
\begin{equation}
m=cp^r+n.
\end{equation} 
Since $gcd(n,p^r)=1$, by Bezout's identity there are $x,y\in\Z$ such that 

\begin{equation}
nx+p^{r}y = 1. 
\end{equation}

Therefore $tat^{-1}=ta^{nx+p^{r}y}t^{-1}$ and hence $ tat^{-1}=ta^{nx}t^{-1}ta^{p^{r}y}t^{-1}$.
		
		Raising the last equation to the power $p^{s-r}$ and using the fact that $a^{p^{s}}=1$ we get $$ta^{p^{s-r}}t^{-1}=ta^{nxp^{s-r}}t^{-1}=(ta^{n}t^{-1})^{xp^{s-r}}.$$ Using $(1)$ we have 
\begin{equation}
ta^{n}t^{-1}=a^{cp^r+n}
\end{equation} 
and hence $$ta^{p^{s-r}}t^{-1}=(a^{cp^r+n})^{xp^{s-r}}=a^{nxp^{s-r}}\cdot a^{xcp^s}.$$ But $a^{xcp^s}=1$ and thus using again the equation $(2)$ we have $$ta^{p^{s-r}}t^{-1}=a^{(1-p^ry)p^{s-r}}=a^{p^{s-r}}\cdot a^{-yp^s}=a^{p^{s-r}}.$$ Therefore, we have $[t,a^{p^{s-r}}]=1$ and thus $a^{p^{s-r}}$ is central in $G$.

Since $gcd(n,p^{s-r})=1$ by Bezout's identity there are $z,w\in\Z$ such that 
\begin{equation}
nz+p^{s-r}w=1. 
\end{equation} 

Hence, $tat^{-1}=ta^{nz}t^{-1}(ta^{p^{s-r}}t^{-1})^{w}$. Since $[t,a^{p^{s-r}}]=1$ we have $$tat^{-1}=(ta^{n}t^{-1})^{z}a^{w p^{s-r}} \myex a^{zcp^{r}+zn}\cdot a^{w p^{s-r}}=a^{zcp^{r}}\cdot a^{zn+wp^{s-r}}=a^{zcp^{r}}\cdot a=a^{1+zcp^r}.$$ 
		
		We will prove that $t^{p^{s-r}}at^{-p^{s-r}}=a$ which implies that $t$ has order a power of $p$. 		
		We know that $t^{p^{s-r}}at^{-p^{s-r}}=a^{(1+zcp^r)^{p^{s-r}}}$.
		It is also known that  $$({1+zcp^r)}^{p^{s-r}}=\sum_{i=0}^{p^{s-r}}{p^{s-r} \choose i}(zcp^{r})^{i}.$$
		
		The first term $(i=0)$ is ${p^{s-r} \choose 0}\cdot 1$ and thus the first term of the sum is equal to $1$. 
		We will  prove that for $i\geq1$  every term ${p^{s-r} \choose i}(zcp^{r})^{i}$ is also a multiple of $p^s$ and thus since $a^{p^s}=1$ we will have $a^{(1+zcp^r)^{p^{s-r}}}=a$. For $i=1$ the second term is ${p^{s-r} \choose 1}\cdot zcp^r =p^{s-r}\cdot zcp^r=zcp^s$. 
		
		Assume that $i\geq 2$. If $i<p$ then $p^{s-r}\mid {p^{s-r} \choose i}$, that is ${p^{s-r} \choose i}=\lambda p^{s-r}$. Hence $\lambda p^{s-r}(zcp^{r})^{i}=p^{s+(i-1)r}\cdot \lambda zc$, with $(i-1)>1$. Thus we may assume that $i\geq p$. Using  Kummer's Theorem for the  binomial coefficient ${p^{s-r} \choose i}$ we get $\nu_{p}{p^{s-r} \choose i}={s-r}-\nu_{p}(i)$.  Therefore, $\nu_{p}({p^{s-r} \choose i}(zcp^{r})^{i})={s-r}-\nu_{p}(i)+ri=s-\nu_{p}(i)+(r-1)i$. But $\nu_{p}(i)<ri$ for all $r,i>0$ and thus $\nu_{p}({p^{s-r} \choose i}(zcp^{r})^{i})>s$.  Consequently, the term $a^{{p^{s-r} \choose i}(zcp^{r})^{i}}$ is again equal to $1$. 
		
		Since $a^{({1+zcp^r)}^{p^{s-r}}}=a$ we have  $t^{p^{s-r}}at^{-p^{s-r}}=a$ and hence the automorphism  induced by $t$ has order a power of $p$. 
	\end{proof}
	\begin{remark*}
		In  \cite{RV} Raptis and Varsos proved that the semi-direct product $G=K\rtimes\langle t \rangle$ of a finite $p$-group $K$ by the infinite cyclic group $\langle t \rangle$ is residually finite $p$ if and only if the automorphism $\phi$ of $K$ induced by $t$ has order a power of $p$. The above lemma is a partial generalization of this result.
	\end{remark*}
	\begin{lemma}\label{corRFp}
		Let $p$ be a prime number. Consider the following multiple HNN-extension of a cyclic $p$-group $$G=\langle a, t_{i} \hspace{1mm}|\hspace{1mm}a^{p^s}=1, t_{i}a^{n_{i}}t_{i}^{-1}=a^{m_{i}}, i\in\{1,\dots,r\} \rangle$$ 
		where   $n_{i}=p^{\sigma_{i}}\hat{n}_{i},m_{i}=p^{\sigma_{i}}\hat{m}_{i}$ for $i\in\{1,\dots,r\}$, $s,\sigma_{i}\in \N$ and $p\nmid \hat{n}_{i}, \hat{m}_{i}$. Assume also that $p|(\hat{m}_{i}-\hat{n}_{i})$ for all $i\in\{1,\dots,r\}$ .
		Then $G$ is a residually finite $p$-group.
	\end{lemma}
	\begin{proof}
		For every $i\in\{1,\dots,r\}$ such that $p\nmid n_{i},m_{i}$ and hence $\sigma_{i}$=0,  $a^{n_{i}}$ and $a^{m_{i}}$ are generators of $\langle a \rangle$, and thus, by Lemma \ref{aut}, we have that $t_{i}$ induces an automorphism with order a power of $p$. 
		
	On the other hand, for the cases where  $\sigma_{i}>0$, Lemma \ref{aut} holds for the subgroup generated by $a^{p^{\sigma_{i}}}$. That means that each automorphism of the subgroup $\langle a^{p^{\sigma_{i}}}\rangle$  induced by $t_{i}$ has order a power of $p$. Therefore, we can extend this to an automorphism for the group $\langle a \rangle$ that maps $a^{\hat{n_{i}}}\mapsto a^{\hat{m_{i}}}$ and has also order a power of $p$. 
		
		Consequently, for all $i\in\{1,\dots,r\}$,  $t_{i}$ induce automorphisms on the group $\<a\>$ that form a finite $p$-group and thus, by Lemma \ref{genRV}, we get that $G$ is residually finite $p$-group.
	\end{proof}
	In the following Lemma we prove an auxiliary result which we need for the proof of our main Theorem.
	
\begin{lemma}\label{dioph}[see \cite{Arturo}]\ Consider the following diophantine equations.
\begin{equation}\label{d1}
a_{1}x_{1}+\dots+a_{n}x_{n}=c
\end{equation}
\begin{equation}\label{d2}
gcd(a_{1},\dots,a_{n-1})y+a_{n}x_{n}=c
\end{equation}
		Then every solution of the diophantine equation (\ref{d1}) yields a solution of the diophantine equation (\ref{d2}) and vice versa.
	\end{lemma}
	\begin{proof}
		Let $g= gcd(a_{1},\dots,a_{n-1})$.  Then,  by Bezout's identity,  there are $r_{1},\dots,r_{n-1}\in\Z$ such that $r_{1}a_{1}+\dots+r_{n-1}a_{n-1}=g$.  Let $y,x_n$ be a solution of (\ref{d2}).  Then  
		$$c=gy+a_{n}x_{n}=(r_{1}a_{1}+\dots+r_{n-1}a_{n-1})y+a_{n}x_{n}=a_{1}(r_{1}y)+\dots+a_{n-1}(r_{n-1}y)+a_{n}x_{n}.$$
		Thus $x_{i}=r_{i}y$, $i\in \{1,\dots,n-1\}$, $x_n$ is a solution of (\ref{d1}). 
		
		Conversely,  let $x_i$, $i=1,\ldots, n$ be a solution of (\ref{d1}).  Since $g= gcd(a_{1},\dots,a_{n-1})$ we can write $a_{1}=gk_{1},\dots,a_{n-1}=gk_{n-1}$ where $k_{1},\dots,k_{n-1}\in\Z$. Then we have 
		$$c=a_{1}x_{1}+\dots+a_{n}x_{n}=gk_{1}x_{1}+gk_{n-1}x_{n-1}+a_{n}x_{n}=g(k_{1}x_{1}+\dots + k_{n-1}x_{n-1})+a_{n}x_{n}.$$
		Thus,  $k_{1}x_{1}+\dots + k_{n-1}x_{n-1}, x_n$ is a solution to (\ref{d2}).
	\end{proof}

	\section{Proof of the main Theorem}\label{loop}

We can now prove our main Theorem which calculates the intersection  $(N_{p})_{\om}(G)$ of the group
$$G=\langle a, t_{1},\dots,t_{r}\hspace{1mm}|\hspace{1mm}t_{i}a^{n_{i}}t_{i}^{-1}=a^{m_{i}}, i\in\{1,\dots,r\}\rangle.$$
We use the notation fixed in the introduction.

\vspace{3mm}

Let $f:$ $G\to P_s$ be a homomorphism, where $P_s$ is a $p$-group of order $p^s$ and let us denote by  $\bar{g}$ the image of any $g\in G$ under $f$, i.e.  $f(g)=\bar{g}$.
\begin{enumerate}

	\item[(1)]  Assume without loss of generality that $\theta_{1}\in \N$.  We have to examine the case  $\sigma_{1}\neq \tau_{1}$ and the case $\sigma_{1}=\tau_{1}$ and $p\nmid(\hat{m}_{1}-\hat{n}_{1})$.

	Assume at first that $\sigma_{1}\neq \tau_{1}$ (and assume $\sigma_{1}<\tau_{1}$). Then in $P_s$ we have  $\bar{t}_{1}\bar{a}^{p^{\sigma_{1}}\hat{n}_{1}}\bar{t}_{1}^{-1}=\bar{a}^{p^{\tau_{1}}\hat{m}_{1}}$ and hence $[\bar{t}_{1},\bar{a}^{p^{\sigma_{1}}\hat{n}_{1}}]=\bar{a}^{p^{\sigma_{1}}(p^{\tau_{1}-\sigma_{1}}\hat{m}_{1}-\hat{n}_{1})}$. But $p\nmid \hat{n}_{1}$ and thus $p\nmid ( p^{\tau_{1}-\sigma_{1}}\hat{m}_{1}-\hat{n}_{1})$. Therefore, for $s\geq \sigma_{1}$ we have $\bar{a}^{p^{\sigma_{1}}}\in\gamma_{2}(P_s)$.   But since $\bar{a}^{p^{\sigma_{1}}(p^{\tau_{1}-\sigma_{1}}\hat{m}_{1}-\hat{n}_{1})}=[\bar{t}_{1},\bar{a}^{p^{\sigma_{1}}\hat{n}_{1}}]$ and $\bar{a}^{p^{\sigma_{1}}}\in\gamma_{2}(P_s)$ we have  $\bar{a}^{p^{\sigma_{1}}(p^{\tau_{1}-\sigma_{1}}\hat{m}_{1}-\hat{n}_{1})}\in\gamma_{3}(P_s)$. Again since $\bar{a}^{p^s}=1$ and $p\nmid ( p^{\tau_{1}-\sigma_{1}}\hat{m}_{1}-\hat{n}_{1})$ we have  $\bar{a^{p^{\sigma_{1}}}}\in\gamma_{3}(P_s)$.  Inductively, we get that  $\bar{a^{p^{\sigma_{1}}}}\in\gamma_{c}(P_s)$, for all $c\in \N$ and thus $\bar{a^{p^{\sigma_{1}}}}\in\gamma_{\om}(P_s)$. But since $P_s$ is a $p$-group and thus residually nilpotent we have  $\bar{a^{p^{\sigma_{1}}}}=1$  and thus  $a^{p^{\sigma_{1}}} =a^{p^{\theta_{1}}}\in (N_p)_{\omega}(G)$. 
	
	Now assume that $\theta_{j}\in\N$ for all $j\in C$, where $C$ is a subset (possibly empty) of $\{2,\dots,r\}$ while the remaining $\theta_{j}$, $j\in\{1,\dots,r\}\setminus C$  are equal to infinity. Using the same argument as above we have  $a^{p^{\theta_{j}}}\in (N_p)_{\omega}(G)$, for all $j\in C$. Since $p^\xi |m_{1},m_{j},n_{1},n_{j}$ we have   $$G/\langle a^{p^{\xi}}\rangle^G \cong F_{|C|+1}*G_{1} $$ where $F_{|C|+1}$ is the free group of rank $|C|+1$ and   $$G_{1}=\langle a, t_{i} \hspace{1mm}|\hspace{1mm} a^{p^{\xi}}=1,  t_{i}a^{n_{i}}t_{i}^{-1}=a^{m_{i}}, i\in \{2,\dots,r\}\setminus C \rangle.$$ 
	
	Since $\sigma_{i}=\tau_{i}$ and $p|(\hat{m}_{i}-\hat{n}_{i})$ for all $ i\in \{2,\dots,r\}\setminus C $, by Lemma \ref{corRFp} we have that $G_{1}$ is residually finite $p$-group and therefore so is the free product $F_{|C|+1}*G_{1}$. Hence, we conclude  that $(N_p)_{\omega}(G)=\langle a^{p^{\xi}}\rangle^G$.
	
	Assume now, for our second case that $\sigma_{1}=\tau_{1}=\theta_{1}$ and $p\nmid(\hat{m}_{1}-\hat{n}_{1})$. Then, in $P_s$, we have $[\bar{t}_{1},\bar{a}^{p^{\theta_{1} \hat{n}_{1}}}]=\bar{a}^{p^{\theta_{1} }(\hat{m}_{1}-\hat{n}_{1})}$. Since $p\nmid(\hat{m}_{1}-\hat{n}_{1})$, using the same argument, we have $a^{p^{\theta_{1}}}\in (N_p)_{\omega}(G)$ and as in the previous case we get that $(N_p)_{\omega}(G)=\langle a^{p^{\xi}}\rangle^G$.
	
	\item[(3)] In order to simplify the proof we assume that $i=3$ and 	thus $\theta_{1}=\theta_{2}=\theta_{3}=\infty$.

	For $s>\Sigma=\sigma_{1}+\sigma_{2}+\sigma_{3}$, Since $gcd(p^{\Sigma} u_{1}u_{2}u_{3},p^{s})=p^{\Sigma}$, 	for $s>\Sigma=\sigma_{1}+\sigma_{2}+\sigma_{3}$,
	there are $x_{s},y_{s}\in\Z$ such that: 
\begin{equation}\label{l1}	
p^{\Sigma}u_{1}u_{2}u_{3}x_{s}+p^{s}y_{s}=p^{\Sigma}
\end{equation}

	Therefore the following holds in $P_{s}$: $$\bar{t}_{1}\bar{a}^{p^{\sigma_{1}}u_{1}}\bar{t}_{1}^{-1}=\bar{a}^{p^{\sigma_{1}}v_{1}}\Rightarrow \bar{t}_{1}\bar{a}^{p^{\Sigma}u_{1}u_{2}u_{3}x_{s}}\bar{t}_{1}^{-1}=\bar{a}^{p^{\Sigma}v_{1}u_{2}u_{3}x_{s}}\xRightarrow{(\ref{l1})} \bar{t}_{1}\bar{a}^{p^{\Sigma}}\bar{t}_{1}^{-1}=\bar{a}^{p^{\Sigma}v_{1}u_{2}u_{3}x_{s}}. $$
	Similarly, we get $\bar{t}_{2}\bar{a}^{p^{\Sigma}}\bar{t}_{2}^{-1}=\bar{a}^{p^{\Sigma}u_{1}v_{2}u_{3}x_{s}} $ and $\bar{t}_{3}\bar{a}^{p^{\Sigma}}\bar{t}_{3}^{-1}=\bar{a}^{p^{\Sigma}u_{1}u_{2}v_{3}x_{s}}$.

	We see that  $\bar{t}_{1}$, $\bar{t}_{2}$ and $\bar{t}_{3}$ act as automorphisms to the subgroup $\langle \bar{a}^{p^{\Sigma}}\rangle$ of $\langle \bar{a} \rangle $ in $P_{s}$. Thus, if $w$ is a word in $\gamma_{2}(\<\bar{t}_{1},\bar{t}_{2},\bar{t}_{3}\>)$ then $w\bar{a}^{p^{\Sigma}}w^{-1}=\bar{a}^{p^{\Sigma}}$, which means that  $$[\gamma_{2}(\<t_{1},t_{2},t_{3}\>),a^{p^{\Sigma}}]\in (N_p)_{\omega}(G).$$

	To simplify the notation, let $k_{1}=\k,k_{2}=\l$ and $k_{3}=\mu$ and assume that $\k,\l,\mu\geq 0$. Then we have that since the following relations hold in $P_{s}$
	$$\bar{t}_{1}^{\k}\bar{a}^{p^{\sigma_{1}}u_{1}^{\k}}\bar{t}_{1}^{-\k}=\bar{a}^{p^{\sigma_{1}}v_{1}^{\k}},\bar{t}_{2}^{\l}\bar{a}^{p^{\sigma_{2}}u_{2}^{\l}}\bar{t}_{2}^{-\l}=\bar{a}^{p^{\sigma_{2}}v_{2}^{\l}}, \bar{t}_{3}^{\mu}\bar{a}^{p^{\sigma_{3}}u_{3}^{\mu}}\bar{t}_{3}^{-\mu}=\bar{a}^{p^{\sigma_{3}}v_{3}^{\mu}}$$
	we have
	$$\bar{t}_{1}^{\k}\bar{t}_{2}^{\l}\bar{t}_{3}^{\mu}\bar{a}^{p^{\Sigma}{u_{1}^{\k}u_{2}^{\l}u_{3}^{\mu}}}\bar{t}_{3}^{-\mu}\bar{t}_{2}^{-\l}\bar{t}_{1}^{-\k}=\bar{a}^{p^{\Sigma}{v_{1}^{\k}v_{2}^{\l}v_{3}^{\mu}}}.$$
	But $u_{1},u_{2},u_{3},v_{1},v_{2},v_{3}$ are coprime with $p$ so
	$$\bar{t}_{1}^{\k}\bar{t}_{2}^{\l}\bar{t}_{3}^{\mu}\bar{a}^{p^{\Sigma}\frac{u_{1}^{\k}u_{2}^{\l}u_{3}^{\mu}}{d_{\k,\l,\mu}}}\bar{t}_{3}^{-\mu}\bar{t}_{2}^{-\l}\bar{t}_{1}^{-\k}=\bar{a}^{p^{\Sigma}\frac{v_{1}^{\k}v_{2}^{\l}v_{3}^{\mu}}{d_{\k,\l,\mu}}}$$
	where ${d_{\k,\l,\mu}}=gcd(u_{1}^{\k}u_{2}^{\l}u_{3}^{\mu},v_{1}^{\k}v_{2}^{\l}v_{3}^{\mu})$.
	Working in the same way we also get
	$$\bar{t}_{1}^{-\k}\bar{t}_{2}^{\l}\bar{t}_{3}^{\mu}\bar{a}^{p^{\Sigma}\frac{v_{1}^{\k}u_{2}^{\l}u_{3}^{\mu}}{e_{\k,\l,\mu}}}\bar{t}_{3}^{-\mu}\bar{t}_{2}^{-\l}\bar{t}_{1}^{\k}=\bar{a}^{p^{\Sigma}\frac{u_{1}^{\k}v_{2}^{\l}v_{3}^{\mu}}{e_{\k,\l,\mu}}}$$
	$$\bar{t}_{1}^{\k}\bar{t}_{2}^{-\l}\bar{t}_{3}^{\mu}\bar{a}^{p^{\Sigma}\frac{u_{1}^{\k}v_{2}^{\l}u_{3}^{\mu}}{f_{\k,\l,\mu}}}\bar{t}_{3}^{-\mu}\bar{t}_{2}^{\l}\bar{t}_{1}^{-\k}=\bar{a}^{p^{\Sigma}\frac{v_{1}^{\k}u_{2}^{\l}v_{3}^{\mu}}{f_{\k,\l,\mu}}}$$
	$$\bar{t}_{1}^{\k}\bar{t}_{2}^{\l}\bar{t}_{3}^{-\mu}\bar{a}^{p^{\Sigma}\frac{u_{1}^{\k}u_{2}^{\l}v_{3}^{-\mu}}{g_{\k,\l,\mu}}}\bar{t}_{3}^{\mu}\bar{t}_{2}^{-\l}\bar{t}_{1}^{-\k}=\bar{a}^{p^{\Sigma}\frac{v_{1}^{\k}v_{2}^{\l}u_{3}^{\mu}}{g_{\k,\l,\mu}}}$$

	where $${e_{\k,\l,\mu}}=gcd(v_{1}^{\k}u_{2}^{\l}u_{3}^{\mu},u_{1}^{\k}v_{2}^{\l}v_{3}^{\mu})$$ $${f_{\k,\l,\mu}}=gcd(u_{1}^{\k}v_{2}^{\l}u_{3}^{\mu},v_{1}^{\k}u_{2}^{\l}v_{3}^{\mu})$$ $${g_{\k,\l,\mu}}=gcd(u_{1}^{\k}u_{2}^{\l}v_{3}^{\mu},v_{1}^{\k}v_{2}^{\l}u_{3}^{\mu}).$$
	
	Therefore we have the elements $${t}_{1}^{\k}{t}_{2}^{\l}{t}_{3}^{\mu}a^{p^{\Sigma}\frac{u_{1}^{\k}u_{2}^{\l}u_{3}^{\mu}}{d_{\k,\l,\mu}}}{t}_{3}^{-\mu}{t}_{2}^{-\l}{t}_{1}^{-\k}a^{-p^{\Sigma}\frac{v_{1}^{\k}v_{2}^{\l}v_{3}^{\mu}}{d_{\k,\l,\mu}}}$$  $${t}_{1}^{-\k}{t}_{2}^{\l}{t}_{3}^{\mu}a^{p^{\Sigma}\frac{v_{1}^{\k}u_{2}^{\l}u_{3}^{\mu}}{e_{\k,\l,\mu}}}{t}_{3}^{-\mu}{t}_{2}^{-\l}{t}_{1}^{\k}a^{-p^{\Sigma}\frac{u_{1}^{\k}v_{2}^{\l}v_{3}^{\mu}}{e_{\k,\l,\mu}}}$$ $${t}_{1}^{\k}{t}_{2}^{-\l}{t}_{3}^{\mu}a^{p^{\Sigma}\frac{u_{1}^{\k}v_{2}^{\l}u_{3}^{\mu}}{f_{\k,\l,\mu}}}{t}_{3}^{-\mu}{t}_{2}^{\l}{t}_{1}^{-\k}a^{-p^{\Sigma}\frac{v_{1}^{\k}u_{2}^{\l}v_{3}^{\mu}}{f_{\k,\l,\mu}}}$$ $${t}_{1}^{\k}{t}_{2}^{\l}{t}_{3}^{-\mu}a^{p^{\Sigma}\frac{u_{1}^{\k}u_{2}^{\l}v_{3}^{-\mu}}{g_{\k,\l,\mu}}}{t}_{3}^{\mu}{t}_{2}^{-\l}{t}_{1}^{-\k}a^{-p^{\Sigma}\frac{v_{1}^{\k}v_{2}^{\l}u_{3}^{\mu}}{g_{\k,\l,\mu}}}$$ in  $(N_p)_{\omega}(G)$. 
	
	Now consider again the relation 
		$$\bar{t}_{1}^{\k}\bar{t}_{2}^{\l}\bar{t}_{3}^{\mu}\bar{a}^{p^{\Sigma}{u_{1}^{\k}u_{2}^{\l}u_{3}^{\mu}}}\bar{t}_{3}^{-\mu}\bar{t}_{2}^{-\l}\bar{t}_{1}^{-\k}=\bar{a}^{p^{\Sigma}{v_{1}^{\k}v_{2}^{\l}v_{3}^{\mu}}}.$$
		
		For $s>\Sigma$, since $gcd(p^{\Sigma}{u_{1}^{\k}u_{2}^{\l}u_{3}^{\mu}},p^{s})=p^{\Sigma}$,
		there are $z_{s},w_{s}\in\Z$ such that: $$p^{\Sigma}u_{1}^{\k}u_{2}^{\l}u_{3}^{\mu}x_{s}+p^{s}y_{s}=p^{\Sigma}, \hspace{3mm}\k,\l,\mu\geq 0.$$
		Thus, by working as before we get
		$$\bar{t}_{1}^{\k}\bar{t}_{2}^{\l}\bar{t}_{3}^{\mu}\bar{a}^{p^{\Sigma}}\bar{t}_{3}^{-\mu}\bar{t}_{2}^{-\l}\bar{t}_{1}^{-\k}=\bar{a}^{p^{\Sigma}{v_{1}^{\k}v_{2}^{\l}v_{3}^{\mu}}z_{s}}.$$
		which means that $\bar{t}_{1}^{\k}\bar{t}_{2}^{\l}\bar{t}_{3}$ act as automorphism to the subgroup $\langle \bar{a}^{p^{\Sigma}}\rangle$ of $\langle \bar{a} \rangle $ in $P_{s}$, for all $\k,\l,\mu\geq 0.$ Working the same way we have $\bar{t}_{1}^{-\k}\bar{t}_{2}^{\l}\bar{t}_{3}^{\mu}$, $\bar{t}_{1}^{\k}\bar{t}_{2}^{-\l}\bar{t}_{3}^{\mu}$ and $\bar{t}_{1}^{\k}\bar{t}_{2}^{\l}\bar{t}_{3}^{-\mu}$ also act as automorphisms to the subgroup $\langle \bar{a}^{p^{\Sigma}}\rangle$ and thus 
		$$[\bar{t}_{1}^{i}\bar{t}_{2}^{j}\bar{t}_{3}^{k}a\bar{t}_{1}^{-i}\bar{t}_{2}^{j}\bar{t}_{3}^{-k},a^{p^{\Sigma}}]=1 \hspace{3mm} \text{ for all }i,j,k\in \Z$$ 
		which means that $$[t_{1}^{i}t_{2}^{j}t_{3}^{k}at_{1}^{-i}t_{2}^{j}t_{3}^{-k},a^{p^{\Sigma}}]\in (N_p)_{\omega}(G).$$
	Let $\bar{G}$ be the factor group $G$ by the normal closure of the set 	\begin{center}
		$\begin{Bmatrix}
			[\gamma_{2}(\<t_{1},t_{2},t_{3}\>),a^{p^{\Sigma}}], [t_{1}^{i}t_{2}^{j}t_{3}^{k}at_{1}^{-i}t_{2}^{j}t_{3}^{-k},a^{p^{\Sigma}}]\\{t}_{1}^{\k}{t}_{2}^{\l}{t}_{3}^{\mu}a^{p^{\Sigma}\frac{u_{1}^{\k}u_{2}^{\l}u_{3}^{\mu}}{d_{\k,\l,\mu}}}{t}_{3}^{-\mu}{t}_{2}^{-\l}{t}_{1}^{-\k}a^{-p^{\Sigma}\frac{v_{1}^{\k}v_{2}^{\l}v_{3}^{\mu}}{d_{\k,\l,\mu}}},
			{t}_{1}^{-\k}{t}_{2}^{\l}{t}_{3}^{\mu}a^{p^{\Sigma}\frac{v_{1}^{\k}u_{2}^{\l}u_{3}^{\mu}}{e_{\k,\l,\mu}}}{t}_{3}^{-\mu}{t}_{2}^{-\l}{t}_{1}^{\k}a^{-p^{\Sigma}\frac{u_{1}^{\k}v_{2}^{\l}v_{3}^{\mu}}{e_{\k,\l,\mu}}},\\ {t}_{1}^{\k}{t}_{2}^{-\l}{t}_{3}^{\mu}a^{p^{\Sigma}\frac{u_{1}^{\k}v_{2}^{\l}u_{3}^{\mu}}{f_{\k,\l,\mu}}}{t}_{3}^{-\mu}{t}_{2}^{\l}{t}_{1}^{-\k}a^{-p^{\Sigma}\frac{v_{1}^{\k}u_{2}^{\l}v_{3}^{\mu}}{f_{\k,\l,\mu}}},
			{t}_{1}^{\k}{t}_{2}^{\l}{t}_{3}^{-\mu}a^{p^{\Sigma}\frac{u_{1}^{\k}u_{2}^{\l}v_{3}^{-\mu}}{g_{\k,\l,\mu}}}{t}_{3}^{\mu}{t}_{2}^{-\l}{t}_{1}^{-\k}a^{-p^{\Sigma}\frac{v_{1}^{\k}v_{2}^{\l}u_{3}^{\mu}}{g_{\k,\l,\mu}}}
		\end{Bmatrix}$
	\end{center}  in $G$,  $i,j,k\in \Z$, $\k,\l,\mu\in \N$.
	\vspace{2mm}
	
	We will prove that $\bar{G}$ is a residually finite $p$-group.
	
	We define the natural epimorphisms $\psi_{s}:{\bar{G}} \twoheadrightarrow \bar{G}/\langle a ^{p^{s}}\rangle^{\bar{G}}=H_{s}$ for all $s>\Sigma$. Notice that $H_{s}$ have the following presentation. 
	$$H_{s}=\langle t_{1},t_{2},t_{3},a\hspace{1mm}|\hspace{1mm} a^{p^{s}}=1, t_{1}a^{p^{\sigma_{1}}u_{1}}t_{1}^{-1}=a^{p^{\sigma_{1}}v_{1}},t_{2}a^{p^{\sigma_{2}}u_{2}}t_{2}^{-1}=a^{p^{\sigma_{2}}v_{2}},t_{3}a^{p^{\sigma_{3}}u_{3}}t_{3}^{-1}=a^{p^{\sigma_{3}}v_{3}}\rangle.$$
	
	Moreover, since $p|(\hat{m}_{i}-\hat{n}_{i})$ for all $i=1,2,3$ we have $p|d_{i}(v_{i}-u_{i})$. Since $gcd(p,d_{i})=1$ we conclude that $p|(v_{i}-u_{i})$. Therefore, by Lemma \ref{corRFp} we have $H_{s}$ are residually finite $p$-groups. Again, using the same argument as above we have the following relations hold in $H_{s}$:
	$$ {t}_{1}{a}^{p^{\Sigma}}{t}_{1}^{-1}={a}^{p^{\Sigma}v_{1}u_{2}u_{3}x_{s}},\hspace{2mm}{t}_{2}{a}^{p^{\Sigma}}{t}_{2}^{-1}={a}^{p^{\Sigma}u_{1}v_{2}u_{3}x_{s}},\hspace{2mm} {t}_{3}{a}^{p^{\Sigma}}{t}_{3}^{-1}={a}^{p^{\Sigma}u_{1}u_{2}v_{3}x_{s}}.$$
	
	To prove that $\bar{G}$ is a residually finite $p$-group, it is sufficient to prove that $\bigcap\limits_{s>\Sigma} Ker\psi_{s}$ is trivial. 
	
	Let $1\neq g\in Ker\psi_{s_{0}}=\langle a^{p^{s_{0}}}\rangle^{\bar{G}} $, for some $s_{0}$ fixed. We will prove that there is some $s>s_{0}$ such that $g\notin Ker\psi_{s}$  and thus the intersection is trivial.
	
	Since $g\in \langle a^{p^{s_{0}}}\rangle^{\bar{G}} $, $g$ can be written in the form, 
	$$g=g_{1}a^{e_{1}p^{s_{0}}}g_{1}^{-1}\cdot g_{2}a^{e_{2}p^{s_{0}}}g_{2}^{-1}\cdots g_{k}a^{e_{k}p^{s_{0}}}g_{k}^{-1}, \hspace{3mm} \text{with } g_{i}\in \bar{G}.$$
	Using the relations that hold in $\bar{G}$ we can rewrite $g$ as follows. 
	$$g=t_{1}^{\k_{1}}t_{2}^{\l_{1}}t_{3}^{\mu_{1}}a^{\e_{1}p^{s_{0}}}t_{3}^{-\mu_{1}}t_{2}^{- \l_{1}}t_{1}^{-\k_{1}}\cdot \cdots t_{1}^{\k_{N}}t_{2}^{\l_{N}}t_{3}^{\mu_{N}}a^{\e_{N}p^{s_{0}}}t_{3}^{-\mu_{N}}t_{2}^{- \l_{N}}t_{1}^{-\k_{N}}.$$
	Since an element is trivial if and only if every conjugate of the element is trivial, we may take appropriate conjugate of $g$ such that $\k_{i},\lambda_{i},\mu_{i}>0$ for all $i\in\{1,\dots,N\}$. Moreover we may assume that $g$ is reduced in the sense that $(\k_{i},\l_{i},\mu_{i})$ are discrete.
	
	Then, using the relations of $H_{s}$ we get $\psi_{s}(g)=a^{p^{s_{0}}\cdot q(x_{s})}$, where $q(x_{s})$ is the following polynomial:
	
	$$ q(x_{s})= \sum\limits_{i=1}^{N}\e_{i}(v_{1}u_{2}u_{3})^{\k_{i}}(u_{1}v_{2}u_{3})^{\l_{i}}(u_{1}u_{2}v_{3})^{\mu_{i}}x_{s}^{\k_{i}+\l_{i}+\mu_{i}}.$$

	It suffices that there exists $s>s_{0}$ such that  $q(x_{s})\not\equiv0\pmod{p^{s-s_{0}}}$. We will prove that for the solutions of this modular congruence the element $g$ becomes trivial in $\bar{G}$. 
	
	Assume that 	$$  \sum\limits_{i=1}^{N}\e_{i}(v_{1}u_{2}u_{3})^{\k_{i}}(u_{1}v_{2}u_{3})^{\l_{i}}(u_{1}u_{2}v_{3})^{\mu_{i}}x_{s}^{\k_{i}+\l_{i}+\mu_{i}}\equiv0\pmod{p^{s-s_{0}}} \text{ for every } s>s_{0} \Leftrightarrow $$
	
\begin{equation}\label{l2}
\sum\limits_{i=1}^{N}\e_{i}v_{1}^{\k_{i}}v_{2}^{\l_{i}}v_{3}^{\mu_{i}}u_{1}^{\l_{i}+\mu_{i}}u_{2}^{\k_{i}+\mu_{i}}u_{3}^{\k_{i}+\l_{i}}x_{s}^{\k_{i}+\l_{i}+\mu_{i}}\equiv0\pmod{p^{s-s_{0}}}
\end{equation}

	We now multiply (\ref{l2}) by $u_{1}^{\k_{\xi}}\cdot u_{2}^{\l_{\zeta}}u_{3}^{\mu_{\theta}}$, where $\k_{\xi}=\max\{\k_{i}\}_{i=1}^{N}$, $\l_{\zeta}=\max\{\l_{i}\}_{i=1}^{N}$ and $\mu_{\theta}=\max\{\mu_{i}\}_{i=1}^{N}$. We know that $u_{1}u_{2}u_{3}x_{s}+p^{s-\Sigma}y_{s}=1$. Thus we have $(u_{1}u_{2}u_{3}x_{s})^{\k_{i}+\l_{i}+\mu_{1}}=(1-p^{s}y_{s})^{\k_{i}+\l_{i}+\mu_{i}}$. Notice that the right-hand side of this equality is equal to $1\pmod{p^{s}}$. Then $(\ref{l2})$ is equivalent to the following:

	$$ \sum\limits_{i=1}^{N}\e_{i}v_{1}^{\k_{i}}v_{2}^{\l_{i}}v_{3}^{\mu_{i}}u_{1}^{\k_{\xi}-\k_{i}}u_{2}^{\l_{\zeta}-\l_{i}}u_{3}^{\mu_{\theta}-\mu_{i}}=0 \hspace{2mm} \Leftrightarrow$$
\begin{equation}\label{l3}
 \sum\limits_{i=1}^{N}\e_{i}v_{1}^{\k_{i}-\k_{\nu}}v_{2}^{\l_{i}-\l_{\sigma}}v_{3}^{\mu_{i}-\mu_{\rho}}u_{1}^{\k_{\xi}-\k_{i}}u_{2}^{\l_{\zeta}-\l_{i}}u_{3}^{\mu_{\theta}-\mu_{i}}=0
 \end{equation}
	
	where $\k_{\nu}=\min\{\k_{i}\}_{i=1}^{N}$, $\l_{\sigma}=\min\{\l_{i}\}_{i=1}^{N}$ and $\mu_{\rho}=\min\{\mu_{i}\}_{i=1}^{N}$.
	\vspace{2mm}

	We define $w_{i}$ to be $w_{i}=v_{1}^{\k_{i}-\k_{\nu}}v_{2}^{\l_{i}-\l_{\sigma}}v_{3}^{\mu_{i}-\mu_{\rho}}u_{1}^{\k_{\xi}-\k_{i}}u_{2}^{\l_{\zeta}-\l_{i}}u_{3}^{\mu_{\theta}-\mu_{i}}$.

	If $\tilde{\Omega}=gcd(w_{1},\dots,w_{N-1})$, then according to Lemma \ref{dioph}, solving the above equation is equivalent to solving the following linear homogenous diophantine equation:
\begin{equation}\label{l4}
\tilde{\Omega}y+\e_{N}w_{N}=0.
\end{equation}
	Let $\Omega=gcd(\tilde{\Omega},w_{N})=gcd(w_{1},\dots,w_{N})$. Then the solutions of (\ref{l4}) are
	$$\e_{N}=-\dfrac{\tilde{\Omega}}{\Omega}\cdot \tau \hspace{3mm}\text{ and }\hspace{3mm}y=\dfrac{w_{N}}{\Omega}\cdot \tau, \hspace{2mm}\tau\in\Z .$$
	Since $\tilde{\Omega}=gcd(w_{1},\dots,w_{N-1})$ there are $r_{1},\dots,r_{N-1}\in\Z$ such that 
\begin{equation}\label{l5}
r_{1}w_{1}+r_{2}w_{2}+\cdots+ r_{N-1}w_{N-1}=\tilde{\Omega}.
\end{equation}
	Therefore the solutions of the original diophantine equations are
	$$\e_{i}=yr_{i} \hspace{1mm}\text{ for all }i\in\{1,\dots,N-1\} \hspace{2mm} \text{ and }\hspace{2mm}\e_{N}=-\dfrac{\tilde{\Omega}}{\Omega}\cdot \tau.$$
	
	We show that for the above $\e_{i}$, the element $g$ is trivial:

	$$g=t_{1}^{\k_{1}}t_{2}^{\l_{1}}t_{3}^{\mu_{1}}a^{\e_{1}p^{s_{0}}}t_{3}^{-\mu_{1}}t_{2}^{- \l_{1}}t_{1}^{-\k_{1}}\cdot \cdots t_{1}^{\k_{N}}t_{2}^{\l_{N}}t_{3}^{\mu_{N}}a^{\e_{N}p^{s_{0}}}t_{3}^{-\mu_{N}}t_{2}^{- \l_{N}}t_{1}^{-\k_{N}}.$$

	But $g$ is trivial if and only if every conjugate of $g$ is trivial. It suffices to show that the following element is trivial.
	$$\prod\limits_{i=1}^{N-1}t_{1}^{\k_{i}-\k_{N}}t_{2}^{\l_{i}-\l_{N}}t_{3}^{\mu_{i}-\mu_{N}}a^{p^{s_{0}}\e_{i}}t_{3}^{\mu_{N}-\mu_{i}}t_{2}^{\l_{N}- \l_{i}}t_{1}^{\k_{N}-\k_{i}}\cdot  a^{\e_{N}p^{s_{0}}}=$$	
	$$\prod\limits_{i=1}^{N-1}t_{1}^{\k_{i}-\k_{N}}t_{2}^{\l_{i}-\l_{N}}t_{3}^{\mu_{i}-\mu_{N}}a^{p^{s_{0}}r_{i}\frac{w_{N}}{\Omega}\cdot \tau}t_{3}^{\mu_{N}-\mu_{i}}t_{2}^{\l_{N}- \l_{i}}t_{1}^{\k_{N}-\k_{i}}\cdot  a^{-p^{s_{0}}\frac{\tilde{\Omega}}{\Omega}\cdot \tau }.$$
	
	We use the relation (\ref{l5}) as follows.

	$$\tilde{\Omega}=r_{1}w_{1}+r_{2}w_{2}+\cdots+ r_{N-1}w_{N-1}\Leftrightarrow$$
\begin{equation}\label{l6}
\dfrac{\tilde{\Omega}}{\Omega}\cdot \tau =r_{1}\dfrac{w_{1}}{{\Omega}}\tau+r_{2}\dfrac{w_{2}}{{\Omega}}\tau+\cdots+ r_{N-1}\dfrac{w_{N-1}}{{\Omega}}\tau
\end{equation}
	
	It is now sufficient to prove that 
	$$t_{1}^{\k_{i}-\k_{N}}t_{2}^{\l_{i}-\l_{N}}t_{3}^{\mu_{i}-\mu_{N}}a^{p^{s_{0}}\frac{w_{N}}{\Omega}}t_{3}^{\mu_{N}-\mu_{i}}t_{2}^{\l_{N}- \l_{i}}t_{1}^{\k_{N}-\k_{i}}=a^{p^{s_{0}}\frac{w_{i}}{\Omega}} \text{ for all } i\in\{1,\dots,N-1\}$$
	Assume that $\k_{i}\geq \k_{N}$, $\l_{i}\geq \l_{N}$ and $\mu_{i}\leq \mu_{N}$.
	
	We know that the following equality holds in $\bar{G}$: 

	$$t_{1}^{\k_{i}-\k_{N}}t_{2}^{\l_{i}-\l_{N}}t_{3}^{\mu_{i}-\mu_{N}}a^{p^{\Sigma}\frac{u_{1}^{\k_{i}-\k_{N}}u_{2}^{\l_{i}-\l_{N}}v_{3}^{\mu_{N}-\mu_{i}}}{d}}t_{3}^{\mu_{i}-\mu_{N}}t_{2}^{\l_{N}-\l_{i}}t_{1}^{\k_{N}-\k_{i}}=a^{p^{\Sigma}\frac{v_{1}^{\k_{i}-\k_{N}}v_{2}^{\l_{i}-\l_{N}}u_{3}^{\mu_{N}-\mu_{i}}}{d}},$$
	where $d=gcd(u_{1}^{\k_{i}-\k_{N}}u_{2}^{\l_{i}-\l_{N}}v_{3}^{\mu_{N}-\mu_{i}},v_{1}^{\k_{i}-\k_{N}}v_{2}^{\l_{i}-\l_{N}}u_{3}^{\mu_{N}-\mu_{i}})$.
	
	\vspace{2mm}
	Let $z=v_{1}^{\k_{N}-\k_{\nu}}v_{2}^{\l_{N}-\l_{\sigma}}v_{3}^{\mu_{i}-\mu_{\rho}}u_{1}^{\k_{\xi}-\k_{i}}u_{2}^{\l_{\zeta}-\l_{i}}u_{3}^{\mu_{\theta}-\mu_{N}}$.

	Then we have  
	$u_{1}^{\k_{i}-\k_{N}}u_{2}^{\l_{i}-\l_{N}}v_{3}^{\mu_{N}-\mu_{i}}\cdot z=w_{N}$,  $v_{1}^{\k_{i}-\k_{N}}v_{2}^{\l_{i}-\l_{N}}u_{3}^{\mu_{N}-\mu_{i}}\cdot z=w_{i}$.
	
	 Moreover, $zd=gcd(w_{N},w_{i})$.
	Hence the following relation holds: 

\begin{equation}\label{l7}	
t_{1}^{\k_{i}-\k_{N}}t_{2}^{\l_{i}-\l_{N}}t_{3}^{\mu_{i}-\mu_{N}}a^{p^{\Sigma}\frac{w_{N}}{gcd(w_{i},w_{N})}}t_{3}^{\mu_{i}-\mu_{N}}t_{2}^{\l_{N}-\l_{i}}t_{1}^{\k_{N}-\k_{i}}=a^{p^{\Sigma}\frac{w_{i}}{gcd(w_{i},w_{N})}}
\end{equation}
	
	But since $\Omega|gcd(w_{i},w_{N})$ there exists $\Omega'\in \N$ such that $gcd(w_{i},w_{N})=\Omega\cdot \Omega'$
	Therefore we have

	$$t_{1}^{\k_{i}-\k_{N}}t_{2}^{\l_{i}-\l_{N}}t_{3}^{\mu_{i}-\mu_{N}}a^{p^{s_{0}}\frac{w_{N}}{\Omega}}t_{3}^{\mu_{N}-\mu_{i}}t_{2}^{\l_{N}- \l_{i}}t_{1}^{\k_{N}-\k_{i}}=$$	
	$$(t_{1}^{\k_{i}-\k_{N}}t_{2}^{\l_{i}-\l_{N}}t_{3}^{\mu_{i}-\mu_{N}}a^{p^{\Sigma}\frac{w_{N}}{gcd(w_{i},w_{N})}}t_{3}^{\mu_{N}-\mu_{i}}t_{2}^{\l_{N}- \l_{i}}t_{1}^{\k_{N}-\k_{i}})^{p^{s_{0}-\Sigma}\cdot \Omega'}\overset{(\ref{l7})}{=}$$
	$$a^{p^{s_{0}}\frac{w_{i}}{gcd(w_{i},w_{N})}\cdot \Omega'}=a^{p^{s_{0}}\frac{w_{i}}{\Omega}}.$$

\par	We work in the same way for the remaining cases. 
	Therefore there exists $s>s_{0}$ such that $q(x_{s})\not\equiv0\pmod{p^{s-s_{0}}}$ and thus we can find an $s$ big enough such that $\psi_{s}(s)$ is not trivial. As a consequence we get that $\bar{G}$ is a residually finite $p$-group and thus we proved our theorem. \qed
\end{enumerate}
\vspace{2mm}
\subsection{Proof of the Corollary}
\vspace{2mm}

We will now prove that $G$ is residually finite $p$-group if and only if either  $n_{i}=m_{i}=p^{\sigma_{i}}$ for all  $i\in \{1,\dots,r\}$ or $p=2$ and $|n_{i}|=m_{i}=2^{\sigma_{i}}$  for all  $i\in \{1,\dots,r\}$.

	Let $p\neq 2$. Consider the case where $t_{i}a^{p^{\sigma_{i}}}t_{i}^{-1}=a^{{p^{\sigma_{i}}}}$ for $i\in\{1,\dots,r\}$. Then $(N_{p})_{\om}(G)$ is trivial since the relations $$[\gamma_{2}(F_{r}),a^{p^{\Sigma}}]=1, [t_{1}^{k_{1}}\cdots t_{r}^{k_{r}}at_{r}^{-k_{r}}\cdots t_{1}^{-k_{1}},a^{p^{\Sigma}}]=1 \text{ and }  t_{1}^{k_{1}}\cdots t_{r}^{k_{r}}a^{p^{\Sigma}} t_{1}^{-k_{1}}\cdots t_{r}^{-k_{r}}=a^{p^{\Sigma}}$$ hold in the group.
	Simillarly, if $p=2$ and  $t_{i}a^{|2^{\sigma_{i}}|}t_{i}^{-1}=a^{{2^{\sigma_{i}}}}$ for all $i\in\{1,\dots,r\}$ then again $(N_{2})_{\om}(G)$ is trivial.

For the converse, we observe that each $\<t_{i},a\>$ is a subgroup of $G$ isomorphic to a Baumslag-Solitar group and thus every $\<t_{i},a\>$ must be residually finite $p$-group. Assume first that $p\neq2$. Then, Corollary \ref{BS} implies that for each $i$, one of the following two cases holds.  Either $m_{i}=1$ and $n_{i}\equiv 1\pmod p$ or $n_{i}=m_{i}=p^{\sigma_{i}}$.
Therefore, it suffices to show that if there exists $i\in\{1,\dots,r\}$ such that  $m_{i}=1$ and $n_{i}\equiv 1\pmod p$ then $G$ is not a residually finite $p$-group.

Let $\k\in\{1,\dots,r\}$ with $t_{\k}a^{n_{\k}}t_{\k}^{-1}=a$,  and $p|(n_{\k}-1$).  Choose $\l\in\{1,\dots,r\}$ with $\l\neq\k$ such that  $\<t_{\l},a\>$ also satisfies the condition $m_{\l}=1$ and $n_{\l}\equiv 1\pmod p$. Then  $G$ contains the Baumslag-Solitar group $$H_{\k,\l}=\<t_{\l}^{-1} t_{\k},a\>=\<t_{\l}^{-1}t_{\k}a^{n_{\k}}t_{\k}^{-1}t_{\l}=a^{n_{\l}}\>$$ that does not satisfy any of the conditions described in Corollary \ref{BS} and hence it is not residually finite $p$-group. Now, since $H_{\k,\l}$ is a subgroup of $G$, $G$ is  not residually finite $p$-group.

Now assume that no such $\l$ exist. Then $\<t_{\l},a\>$  must satisfy the second  condition  of Corollary \ref{BS},  which  means that $t_{\l}a^{p^{\sigma_{\l}}}t_{\l}^{-1}=a^{p^{\sigma_{\l}}}$ for all $\l\in \{1,\ldots r\}$ with $\l\neq\k$. Then consider the Baumslag-Solitar subgroup $$K_{\k,\l}=\<t_{\l} t_{\k},a\>=\<t_{\l}t_{\k}a^{p^{\sigma_{\l}}n_{\k}}t_{\k}^{-1}t_{\l}^{-1}=a^{p^{\sigma_{\l}}}\>.$$
Again $K_{\k,\l}$ does not satisfy the conditions of Corollary \ref{BS} and hence it is not residually finite $p$-group. Thus, $G$ is  not residually finite $p$-group.
Therefore, if $p\neq2$ and there exists $i\in\{1,\dots,r\}$ such that  $m_{i}=1$ and $n_{i}\equiv 1\pmod p$ we have $G$ is not residually finite $p$-group.

Now assume that $p=2$. Then, Corollary \ref{BS} implies that for each $i$ either $m_{i}=1$ and $n_{i}\equiv 1\pmod 2$ or $n_{i}=-m_{i}=-2^{\sigma_{i}}$ or $n_i=m_i=2^{\sigma_{i}}$.
Again,  it suffices to show that if there exists $i\in\{1,\dots,r\}$ such that  $m_{i}=1$ and $n_{i}\equiv 1\pmod 2$ then $G$ is not a residually finite $p$-group.

Assume that for some ${\k}\in\{1,\dots,r\}$ we have $t_{{\k}}a^{n_{{\k}}}t_{{\k}}^{-1}=a$, with $2|(n_{{\k}}-1$). Consider now some ${\l}\in\{1,\dots,r\}$ for which $\<t_{{\l}},a\>$ is also a residually finite $p$-group. 

If  $\<t_{{\l}},a\>$ satisfies the condition $m_{{\l}}=1$ and $n_{{\l}}\equiv 1\pmod 2$ or $t_{{\l}}a^{2^{\sigma_{{\l}}}}t_{{\l}}^{-1}=a^{2^{\sigma_{{\l}}}}$ then we can work as in case $p\neq2$.  Thus, we may assume that there exist $\l\in\{1,\ldots,r\}$ with $\l\neq k$ such that $t_{{\l}}a^{-2^{\sigma_{{\l}}}}t_{{\l}}^{-1}=a^{2^{\sigma_{{\l}}}}$.  Consider the Baumslag-Solitar subgroup $$M_{{\k},{\l}}=\<t_{{\l}} t_{{\k}},a\>=\<t_{{\l}}t_{{\k}}a^{-n_{{\k}}2^{\sigma_{{\l}}}}t_{{\k}}^{-1}t_{{\l}}^{-1}=a^{2^{\sigma_{{\l}}}}\>.$$
Then $M_{\k,\l}$ does not satisfy any of the conditions described in Corollary \ref{BS} and hence it is not residually finite $p$-group. Thus, $G$ is  not residually finite $p$-group.
Therefore, if $p=2$ and there exists $i\in\{1,\dots,r\}$ such that  $m_{i}=1$ and $n_{i}\equiv 1\pmod p$ then $G$ is not a residually finite $p$-group. Thus the result is proven.


\vspace{5mm}
\begin{thebibliography}{99}  
		\bibitem{BS} G. Baumslag and D. Solitar, Some two-generator one-relator non-Hopﬁan groups, \textit{Bull.Amer. Math. Soc.}, 68 (1962), 199--201
	
	\bibitem{Delga}  A. L. Delgado, D. J. S. Robinson, M. Timm, Cyclic normal subgroups of generalized Baumslag-Solitar groups,  \textit{Comm. Algebra} 45 (4) (2017), 1808--1818
	
	
	\bibitem{Fore} M. Forester,  Splittings of generalized Baumslag-Solitar groups, \textit{Geom. Dedicata}, Volume 121 (2006),  43--59 
	
	\bibitem{Gruen} K.W. Gruenberg, residual properties of infinite soluble groups, {\it Proc. LMS} 7 (3) (1957), 29--62.	
	
	     
	  \bibitem{kimc} G. S. Kim and J. McCarron, On Amalgamated Free Products of Residually p-Finite Groups, {\it J. Algebra}  162 (1) (1993), 1--11.
	
	\bibitem{Kum} E. Kummer, Über die Ergänzungssätze zu den allgemeinen Reciprocitätsgesetzen, {\it Journal für die reine und angewandte Mathematik} 44 (1852),  93-–146
	
	\bibitem{LevittA}  G. Levitt, Quotients and subgroups of Baumslag-Solitar groups, \textit{J. Group Theory} 18 (1) (2015),	1--43.
	
	

	\bibitem{Mold} D.I. Moldavanskii, The intersection of the subgroups of finite index
	in Baumslag–Solitar groups, {\it Mathematical Notes}, { 87} No. 1 (2010), 88--95.
	
	\bibitem{Mold-ru} D.I. Moldavanskii, The intersection of the subgroups of finite $p$-index in Baumslag–Solitar groups, {\it Vestnik Ivanovo State Univ. Ser. Natural, Social Sci.}, { 2} (2010) 106--111. (Russian)
	

	
	\bibitem{RV} E. Raptis, D. Varsos, The residual nilpotence of HNN-extensions with base group a finite or a f.g. abelian group,  {\it J. of Pure and Applied Algebra}  76 (1991), 167--178
	
	\bibitem{Robi} D. J. S. Robinson, Recent results on generalized Baumslag-Solitar groups, {\it Note Mat}. 30 (1) (2010), 37--53.
	
\bibitem{Arturo} https://math.stackexchange.com/questions/20906/how-to-find-an-integer-solution-for-general-diophantine-equation-ax-by-cz/20944\#20944
	
	
\end{thebibliography}
\end{document}